\documentclass[12pt]{article}
\usepackage{latexsym, amsfonts, amscd, amssymb, verbatim, amsmath, enumerate}
\usepackage[colorlinks=true,linkcolor=magenta,citecolor=blue]{hyperref}

\newtheorem{theorem}{Theorem}[section]

\newtheorem{lemma}{Lemma}[section]

\newtheorem{example}{Example}[section]

\newcommand{\proofbox}{\hspace{\fill}{$\Box$}}
\newenvironment{proof}{\textbf{Proof}.}{\proofbox}

\date{}

\def\emptyset{\mbox{{\rm \O}}}
\def\bar{\overline}

\date{30/12/2020}

\begin{document}

\title{ Upper semicontinuity of the solution map to a parametric boundary optimal control problem with unbounded constraint sets}

\author{
Nguyen Hai Son \footnote{\textit{Corresponding author}; School of Applied Mathematics and Informatics, Hanoi University of Science and Technology, No.1 Dai Co Viet, Hanoi, Vietnam; Email: son.nguyenhai1@hust.edu.vn} \ and Tuan Anh Dao \footnote{School of Applied Mathematics and Informatics, Hanoi University of Science and Technology, No.1 Dai Co Viet road, Hanoi, Vietnam; 
	Institute of Mathematics, Vietnam Academy of Science and Technology, No.18 Hoang Quoc Viet road, Hanoi, Vietnam; Email: anh.daotuan@hust.edu.vn}}
\maketitle

\medskip

\noindent {\bf Abstract.} {\small We would like to study the solution stability of a parametric control problem governed by semilinear elliptic equations with a mixed state-control constraint, where the cost function is nonconvex and the admissible set is unbounded. The main goal of this paper is to give some sufficient conditions under which the solution map is upper semicontinuous and continuous with respect to parameters.}

\medskip

\noindent {\bf Key words.} Parametric optimal control, solution stability, upper semicontinuity, boundary control, mixed pointwise constraint.

\noindent {\bf AMS Subject Classifications.} 49K20 $\cdot$ 49K40 $\cdot$ 90C31.

\section{Introduction}

In this paper, let us consider the parametric optimal control problem as follows:

Let $\Omega$ be a bounded domain in $\mathbb{R}^2$ with the boundary $\Gamma$ of class $C^{1,1}$. We need to find a control function $u \in L^2(\Gamma)$ and a corresponding state function $y\in H^1(\Omega) \cap C(\bar \Omega)$, which 
\begin{align}
\text{minimize}\  J(y, u, \mu):=\int_\Omega L(x, y(x),\mu^{(1)}(x))dx + \int_{\Gamma} (\ell(x,y,\mu^{(2)})+\varphi(\mu^{(2)})u^2(x))d\sigma , \label{P1}
\end{align}
subject to
\begin{align}
&\begin{cases}
Ay  + f(x,y)= 0  \quad &{\rm in}\ \Omega,\\
\partial_{n_A} y = u +\lambda^{(1)} \quad  &{\rm on}\ \Gamma,
\end{cases} \label{P2} \\
& g(x,y)+u(x)+ \lambda^{(2)} \leq 0\ {\rm a.a.}\ x\in\Gamma,
\label{P3}
\end{align}
where $L: \Omega \times \mathbb R \times \mathbb R  \to \mathbb R$, $\ell :
\Gamma \times \mathbb R \times \mathbb R \times \mathbb R \to \mathbb R$,
$f: \Omega \times \mathbb R \to \mathbb R$ and $g: \Gamma \times \mathbb R \to \mathbb R$ are functions, $(\mu, \lambda) \in (L^\infty(\Omega)\times L^\infty(\Gamma)) \times (L^2(\Gamma))^2$ is a vector of parameters
 with $\mu=(\mu^{(1)}, \mu^{(2)})$ and $\lambda=(\lambda^{(1)}, \lambda^{(2)})$.
The second-order elliptic operator $A$ is determined by the form
$$ Ay(x)=-\sum_{i,j=1}^2 D_j(a_{ij}(x)D_iy(x))+a_0(x)y(x),  $$ and 
$$
 \partial_{n_A} y(x) = \sum_{i,j=1}^2 a_{ij}(x)D_iy(x)\nu_j(x),
$$
in which $\nu(x)=(\nu_1(x),\nu_2(x))$ denotes the unit outward normal vector to $\Gamma$ at the point $x$. As usual, the measure on the boundary $\Gamma$ is the $1$-dimensional measure induced by the parameterization. 

For simplicity, let us put
$$Y:=H^1(\Omega) \cap C(\bar \Omega),\ U:=L^2(\Gamma),\ \Pi:=L^\infty(\Omega)\times L^\infty(\Gamma),\ \Lambda:=(L^2(\Gamma))^2$$
and for each $\lambda \in \Lambda$, we define
\begin{align*} 
K(\lambda)= \big\{ (y,u)\in Y \times U : \, \eqref{P2}\ \text{and} \ \eqref{P3}\ \text{are satisfied}\big\}.
\end{align*}
Then, the problem \eqref{P1}-\eqref{P3} can be rewritten in the following form:
\begin{align*} 
P(\mu, \lambda)  \quad \begin{cases}
J(y,u, \mu) \to \text{inf}, \\
(y,u)\in K(\lambda).
\end{cases}
\end{align*}

Throughout this paper, we denote by $S(\mu,\lambda)$ the solution set of problem $P(\mu,\lambda)$ corresponding to the point $(\mu,\lambda)$. Furthermore, $(\bar \mu, \bar \lambda)$ and $P(\bar \mu, \bar \lambda)$ stand for the reference point and the unperturbed problem, respectively.

We are interested in investigating the behavior of $S(\mu, \lambda)$ when $(\mu, \lambda)$ varies
around $(\bar\mu, \bar\lambda)$. As we can see, the stability of solution map for nonlinear optimal control problems plays an important role in parametric estimation problems. In particular, it  guarantees that the solution set of perturbed problems is not so far away from that of the unperturbed problems. In the last decade, there have been numerous results in terms of the study of solution stability for parametric optimal control problems governed by ordinary differential equations and partial differential equations. For example, we want to refer the interested readers to \cite{Alt, Dontchev,Griesse,Kien1, KTWY,Kien2, KY,Malanowski1,  Malanowski2,Malanowski3,Malanowski4,Malanowski5, Malanowski6, Nhu1, Son1, Son2} and the references therein to recognize the close connections to our problem.

It is known that if the objective function $J( \cdot, \cdot, \mu)$ is strongly convex for all $\mu$ and the admissible set $K(\lambda)$ is convex, then we may claim that the solution set $S(\mu, \lambda)$ is singleton (see more \cite{Alt}, \cite{Dontchev}, \cite{Griesse}). Moreover, in this case Dontchev \cite{Dontchev} and Malanowski \cite{Malanowski1}-\cite{Malanowski6} showed that under some certain conditions, the solution map is Lipschitz continuous with respect to parameters. Inversely, when above conditions are invalid, the solution map may not be single-valued (see, for instance, \cite[Example 4.1]{Kien1} and \cite[Example 4.2]{Son1} ). Later, thanks to using tools of set-valued analysis and variational analysis, Kien et al. \cite{Kien1, KTWY} succeeded to obtain the lower semicontinuity of the solution map to a parametric optimal control problem for the case, where the objective function $J( \cdot, \cdot, \mu)$ is convex in both the variables and the admissible set $K(\lambda)$ are also convex. More recently, such kind of these results have been extended in \cite{KY, Son1, Son2} when one wants to consider that the objective function may not be convex in both the variables and the admissible set is not convex. Namely, Kien et al. \cite{KY} have studied the problem governed by ordinary differential equations, meanwhile the recent two papers of Son and his collaborator \cite{Son1, Son2} are devoted to the investigation of parametric elliptic optimal control problems. More precisely, they have obtained some results on  the upper semicontinuity and continuity of the solution map. Nevertheless, in the latter papers \cite{Son1,Son2}, the considered constraint is controlled by the form  
\begin{align} \label{constraint3} a(x)\leq g(x,y) +u(x)+\lambda^{(2)}\leq b(x), \quad  \forall x\in \Gamma.
	\end{align}
As a consequence, $K(\lambda_n)$ is uniform bounded for parameters $\lambda_n$. This implies that the admissible set $K(\lambda)$ has the following property: {\it If $\{\lambda_n \}$ is a sequence converging strongly to $\hat \lambda$ in $\Lambda$ then for any sequence $\{(y_n, u_n)\}$, $(y_n, u_n)\in K(\lambda_n)$, there exist a subsequence $\{(y_{n_k}, u_{n_k})\}$ and $(\hat y, \hat u) \in K(\hat \lambda)$ such that $ y_{n_k}$ converges strongly to $\hat y$ and $u_{n_k}$ converges weakly to $\hat u$}" (see more in detail \cite[Lemma 2.3]{Son1} and \cite[Lemma 3.5]{Son2}), which comes into play in the proof of the main theorems there. However, the following example will indicate that the property is no longer valid if the constraint \eqref{constraint3} is replaced by the constraint \eqref{P3}.

\begin{example} {\rm Let $\Omega:=\{x=(x_1,x_2)\in \mathbb{R}^2: \ x^2_1+x^2_2< 1\}$. We consider the state equation as follows:
	\begin{align}
		&\begin{cases}
			-\Delta y  + y+y^3= 0  \quad &{\rm in}\ \Omega,\\
			\partial_{n} y = u +\lambda^{(1)} \quad  &{\rm on}\ \Gamma,
		\end{cases} \label{VD1} \\
		& u(x)+ \lambda^{(2)} \leq 0\ {\rm a.a.}\ x\in\Gamma. 
		\notag
	\end{align}
For any positive integers $n$, we take $\lambda_n = \hat \lambda=(0;0) $, $u_n=-n$, and assume that $y_n$ is a unique solution to  the equation \eqref{VD1} corresponding to $u_n$ and $\lambda_n$. Then, we may conclude that $\lambda_n$ converges strongly to $ \hat \lambda$ in $\Lambda$ and $(y_n, u_n)\in K(\lambda_n)$. Moreover, it holds $\lim_{n\to \infty}\|u_n\|_U=\lim_{n\to \infty}(n\sqrt{|\Gamma|})=\infty $, where $|\Gamma|$ is the volume of $\Gamma$. Hence, ${u_n}$ is unbounded in $U$ and so every subsequence of ${u_n}$ is also unbounded. If there exists a subsequence $\{u_{n_k}\}$ of $\{u_n\}$ such that $u_{n_k}$ converges weakly to $\hat u$, then $\{u_{n_k}\}$ is bounded in $U$, which leads to a contradiction. Consequently, all subsequences of $\{(y_{n},u_{n})\}$ should be not weakly convergent in $Y\times U$.}   
\end{example}

In order to overcome this difficulty as well as establish the upper semicontinuity and the continuity of the solution map $S(\mu,\lambda)$ to the problem \eqref{P1}-\eqref{P3}, our main goal of this paper is to develop the method used in \cite{KY}. For this purpose, we need to impose the coercive condition for the objective function (see the assumption (A2), later). The essential point to emphasize here is that the obtained results in \cite{KY} are of the problem subject to control constraints without these parameters. The other thing worthy of mentioning is the use of a common control variable for all parameters as we may realize in \cite{KY}, in particular, the assumption (H6) and the proof of Theorem 2.1. From this observation, we state that the previous approaches seem invalid in our interest of the present problem \eqref{P1}-\eqref{P3}.  

\textbf{The organization of our paper is as follows:} Assumptions and statement of main results are presented in Section \ref{Sec.2}. Section \ref{Sec.3} is devoted to some auxiliary results, consisting of the existence, uniqueness and convergence of solution to the Robin problems and some properties of the admissible set $K(\lambda)$. Finally, the proof of the main result is provided in Section \ref{Sec.4}.

\section{Assumptions and statement of main result} \label{Sec.2}

In this section, let us impose some assumptions for the problem \eqref{P1}-\eqref{P3}.
Fix $(\bar \mu, \bar \lambda) \in \Pi \times \Lambda$ and a constant
$\epsilon_0 >0$. Notation $h_z$ stands for the derivative with respect to $z$
of a given function $h$.

\noindent (\textbf{A1}) $\,\,L:\Omega\times \mathbb R \times \mathbb R \to
\mathbb R$ and $\ell:\Gamma \times \mathbb R\times \mathbb R \times
\mathbb R\to \mathbb R$ are Carath\'{e}odory functions such that $y
\mapsto L(x,y,\mu^{(1)})$ and $(y,u)\mapsto \ell(x',y,u,\mu^{(2)})$ are
Fr\'{e}chet continuous differentiable functions for a.a. $x\in \Omega$
and $x'\in \Gamma$, respectively, and for all $\mu^{(1)},
\mu^{(2)}\in \mathbb{R}$ with $| \mu^{(1)}-\bar \mu^{(1)}(x)|+|
\mu^{(2)}-\bar \mu^{(2)}(x')| \leq \epsilon_0$. Furthermore, for
each $M>0$ there exist positive numbers $C_{LM}$, $C_{\ell M}$ and
functions $r_{1M} \in L^1(\Omega)$, $r_{2M} \in L^1(\Gamma)$,
$r_{3M} \in L^\infty(\Omega)$, $r_{4M} \in L^\infty(\Gamma)$ such
that
\begin{align*}
&| L(x, y,\mu^{(1)}) | \leq r_{1M}(x),& &| \ell(x', y,\mu^{(2)}) | \leq r_{2M}(x'), \\
&|L_y(x,y,\mu^{(1)})|\leq  r_{3M}(x),& &|L_y(x,y_1, \mu^{(1)})-L_y(x,y_2, \mu^{(1)})|\leq C_{LM}|y_1-y_2|,
\end{align*}
and
\begin{align*} 
|\ell_y(x',y,\mu^{(2)})|\leq r_{4M}(x'),\quad |\ell_y(x', y_1, \mu^{(2)})- \ell_y(x', y_2, \mu^{(2)})| \leq C_{\ell M}|y_1-y_2|,
\end{align*}
for a.a. $x \in \Omega$, $x' \in \Gamma$, for
all $\mu^{(1)},\ \mu^{(2)}$, $y,y_i\in\mathbb R$ satisfying
$|\mu^{(1)}-\bar \mu^{(1)}(x)|+| \mu^{(2)}-\bar \mu^{(2)}(x')| \leq \epsilon_0 $
and $|y|,|y_i| \leq M$ with $ i=1,2$.

\noindent (\textbf{A2}) $\,\,$There exist functions $r_{1} \in L^1(\Omega)$, $r_{2} \in L^1(\Gamma)$ to verify
 $$r_1(x)\leq L(x,y,\mu^{(1)}), \quad r_2(x')\leq \ell(x',y,\mu^{(2)})$$
 for a.a. $x \in \Omega$, $x' \in \Gamma$, for 
 all $\mu^{(1)},\ \mu^{(2)}$ satisfying
$|\mu^{(1)}-\bar \mu^{(1)}(x)|+| \mu^{(2)}-\bar \mu^{(2)}(x')| \leq \epsilon_0. $

\noindent (\textbf{A3}) $\,\,$The function $\varphi: \mathbb R \to \mathbb R$ is continuous and $\varphi(t)>\gamma >0$ for all $t\in \mathbb R$. Moreover, there exist constants $\theta>0$ and $k_\varphi >0$ so that it holds
$$|\varphi(t)-\varphi(s)|\leq k_\varphi |t-s|^\theta, \quad \forall t, s \in \mathbb R.$$

\noindent (\textbf{A4}) $\,\,$The coefficients $a_{ij} \in C^{0,1}(\bar \Omega)$ satisfy $a_{ij}(x)=a_{ji}(x)$. We assume that $a_0 \in L^\infty(\Omega),\ a_0(x)\geq 0$ for a.a. $x\in \Omega$, and $a_0 >0$ on a positive set. In addition, there exists $m_0 >0$ satisfying
$$ m_0 \Vert \xi \Vert^2 \leq \sum_{i,j=1}^2 a_{ij}\xi_i\xi_j, \  \ \forall \xi=(\xi_1, \xi_2) \in \mathbb{R}^2  \quad \text{for a.a.} \quad x\in \Omega. $$

\noindent (\textbf{A5}) $\,\,f: \Omega \times \mathbb R \to \mathbb R$ and $g: \Gamma \times \mathbb R \to \mathbb R$ are  Carath\'{e}odory functions of class $C^1$ with respect to the second variable and satisfy the following properties:
\begin{align*}
&f(\cdot, 0)=0, \ f_{y}(x, y)\geq 0\quad {\rm a.a.}\ x\in\Omega, \\
&g(\cdot, 0)=0, \ g_{y}(x', y)\geq 0\quad {\rm a.a.}\ x' \in \Gamma.
\end{align*}
Moreover, for each $M>0$ there exist two constants $C_{fM}, C_{gM} >0$ such that
\begin{align*}
&\big|f_y(x, y)\big| \leq C_{fM},& &\big|f_{y}(x, y_1)-f_{y}(x, y_2)\big| \leq C_{fM}|y_1-y_2|,  \\
&\big|g_y(x', y)\big| \leq C_{gM},& &\big|g_{y}(x', y_1)-g_{y}(x', y_2)\big| \leq C_{gM}|y_1-y_2|,
\end{align*}for a.a. $x \in \Omega$, $x' \in \Gamma$ and for all $y, y_1, y_2 \in \mathbb{R}$ fulfilling $|y|, |y_1|,|y_2|\leq M.$

Notice that assumption (A1) ensures that $J(\cdot, \cdot,\mu)$ is of class $C^1$ for all $\mu \in \Pi$ satisfying $\|\mu-\bar \mu\|_{\Pi}\leq \epsilon_0$. The assumptions (A2) and (A3) guarantees that not only $J(\cdot, \cdot,\cdot)$ is weakly lower semicontinuous but also every sequence of optimal controls has a strongly convergent subsequence. Finally, under the assumptions (A4)-(A5), the equation \eqref{P2} has a unique solution $y \in Y$ for each $u \in U$ and $\lambda \in \Lambda$.
		
Recall that a multifunction $F: E_1
\rightrightarrows E_2$ between topological spaces is said to be {\it
	lower semicontinuous} at $v_0 \in E_1$ if for any open set $W$ in
$E_2$ satisfying $F(v_0) \cap W \neq \emptyset$, there exists a
neighborhood $V_0$ of $v_0$ such that $F(v) \cap W \neq \emptyset$
for all $v \in V_0$. The multifunction $F$ is said to be {\it upper semicontinuous}
at $v_0 \in E_1$ if for any open set $W$ in $E_2$ satisfying
$F(v_0) \subset W $, there exists a neighborhood $V_0$ of $v_0$
such that $F(v) \subset W$ for all $v \in V_0$. If $F$ is both
lower semicontinuous and upper semicontinuous at $v_0$ then $F$
is said to be {\it continuous} at $v_0$ (see \cite{Aubin1}).

We can now formulate the main result.
\begin{theorem} \label{Theo1}
Suppose that the assumptions (A1)-(A5) are fulfilled. Then, the following assertions are valid:
\begin{itemize}
\item[{\rm (i)}] $S(\mu, \lambda) \neq \emptyset$ for all $(\mu, \lambda) \in \Pi \times \Lambda$;
\item[{\rm (ii)}] $S: \Pi \times \Lambda \to Y \times U$ is upper semicontinuous at $(\bar \mu, \bar \lambda)$.

Moreover, if $S(\bar \mu, \bar \lambda)$ is singleton, then $S(\cdot, \cdot)$ is continuous at $(\bar \mu, \bar \lambda)$.
\end{itemize}
\end{theorem}

At the end of this section, we would like give an example satisfying all assumptions (A1)-(A5).

\begin{example}
	{ \rm Let $\Omega$ be the open unit ball in $\mathbb R^2$ with boundary $\Gamma$. We consider the following problem:
		\begin{align*}
			J(y, u, \mu)=&\int_\Omega \left(y^4(x)-y^2(x)+\mu_1(x)\right)dx \\
			&+ \int_\Gamma \Big ( y^2(x)+y(x)|y(x)|+\sqrt{1+\mu^2_2(x)}u^2(x)\Big)d\sigma \to {\rm inf} \\
			\text{subject to}\quad
			&\begin{cases}
				-\Delta y +y + y^3 = 0\quad &{\rm in}\ \Omega,\\
				\partial_n y   =u+\lambda_1   \quad &{\rm on}\ \Gamma,
			\end{cases} \\
			& y(x)+u(x)+\lambda_2(x) \leq 0 \quad \text{a.a.} \ x\in \Gamma.
		\end{align*}
		Taking $\bar \mu \equiv 0$ and $\bar \lambda \equiv 0$, we have
		\begin{align*}
			& L(x,y,\mu_1)=y^4-y^2+\mu_1,& &\ell(x,y,\mu_2) = y^2-y|y|, \\
			&L_y(x,y,\mu_1)=4y^3-2y,& &\ell_y(x,y,\mu_2)= 2y+2|y|,  \\
			&f(x, y) = y^3,\ g(x,y)=y,& & f_y(x,y)= 3y^2,\ g_y(x,y)=1.
		\end{align*}
		The fact is that $J$ is convex with respect to $u$, but it is not convex with respect to $y$. Moreover, $J$ is not twice differentiable. Now let us show that all assumptions (A1)-(A5) are fulfilled.

At first, it is clear that the assumptions (A4) and (A5) hold.

For assumption (A1), direct calculations give the following estimates:
		\begin{align*}
			&|L(x,y,\mu_1)| \leq  |y|^4+ |y|^2+|\mu_1| \leq  M^4+M^2+\epsilon_0, \\
			&|\ell(x',y,\mu_2)|\leq 2|y|^2\leq 2M^2 ,\\
			&|L_y(x,y,\mu_1)|=|4y^3-2|y||\leq 4|y|^3+2|y|^2\leq 4M^3+2M, \\
			&|L_y(x,y_1,\mu_1)-L_y(x,y_2,\mu_1)|=|2(y_1-y_2)(2y^2_1+2y^2_2+2y_1y_2-1)|\leq 2(1+6M^2)|y_1-y_2|, \\
			&|\ell_y(x',y,\mu_2)| \leq 4|y|\leq 4M,  \\
			&|\ell_y(x', y_1,\mu_2)- \ell_y(x', y_2, \mu_2)|\leq 2|y_1-y_2|+2||y_1|-|y_2||\leq 4|y_1-y_2|
		\end{align*}
		for a.a. $x\in \Omega,\ x' \in \Gamma$, for all $\mu_1, \mu_2, y,y_1,y_2\in\mathbb R$ satisfying $|y|,|y_1|,|y_2|\leq M$ and $|\mu_1|+|\mu_2| \leq \epsilon_0$. All in all, we have claimed the assumption (A1).
		
		For the assumption (A2), we derive
		\begin{align*}
			&L(x, y, \mu_1)=(y^2-1)^2 +y^2 +\mu_1-1 \geq -\epsilon_0-1,\\
			&\ell(x',y,\mu_2)=(|y|+y)|y|\geq 0 
		\end{align*}
		for a.a. $x' \in \Gamma$, for all $\mu_1,\mu_2, y,u\in\mathbb R$ such that $|\mu_1|+|\mu_2|\leq \epsilon_0$. Thus, the assumption (A2) is indicated.
		
		For assumption (A3), we can compute
		\begin{align*}
			\varphi(t)-\varphi(s)&=|\sqrt{1+t^2}-\sqrt{1+s^2}|=\left|\frac{t^2 -s^2}{\sqrt{1+t^2}+\sqrt{1+s^2}}\right|\\
			&=\frac{|t +s|}{\sqrt{1+t^2}+\sqrt{1+s^2}}|t-s|\leq |t-s|
		\end{align*}
		for all $s,t\in \mathbb R$. Therefore, the assumption (A3) is satisfied. This completes our verification.}
		
\end{example}

\section{Auxiliary results} \label{Sec.3}
For any $\vartheta \in L^2(\Omega)$ and $\psi \in L^2(\Gamma)$, we say that $\xi \in H^1(\Omega)$ is a (weak) solution to the equation
\begin{align}
\begin{cases}
Ay  + h(x,y)= \vartheta  \quad &{\rm in}\ \Omega,\\
\partial_{n_A} y +k(x,y)= \psi \quad  &{\rm on}\ \Gamma,
\end{cases} \label{P5}
\end{align}
if it holds
\begin{multline*}
\int_{\Omega}\sum_{i,j=1}^N a_{ij}(x)D_i \xi (x)D_jv(x)dx+\int_{\Omega}a_0(x)\xi(x)v(x)dx+\int_{\Omega}h(x,\xi(x))v(x)dx\\
=\int_{\Omega}\vartheta(x)v(x)dx+\int_{\Gamma}(\psi(x)-k(x,\xi(x))) \tau v(x)d\sigma
\end{multline*}
for all $v \in H^1(\Omega)$, where $\tau: H^1(\Omega)\to L^2(\Gamma)$ is the trace operator (see, for instance, \cite{Troltzsch}). Both the existence and the regularity of solutions to the equation \eqref{P5} are established by the following lemma whose detailed proof could be found in \cite[Theorem 3.1]{Casas0} and \cite[Theorem 4.7]{Troltzsch}.
\begin{lemma} \label{lm1}
Assume that the assumption {\rm (A4)} is fulfilled and functions $h, k$ satisfy the assumption {\rm (A5)} formulated for $f,g$. Then, for each $\vartheta \in L^2(\Omega)$ and $\psi \in L^2(\Gamma)$ the equation \eqref{P5} has a unique solution $\xi \in H^1(\Omega)\cap C(\bar \Omega)$ and there exists a positive constant $C_0$ independent of $\vartheta$, $\psi$, $h$, $k$ such that
\begin{align*}
\Vert \xi\Vert_{H^1(\Omega)}+\Vert \xi \Vert_{C(\bar \Omega)} \leq C_0 (\|\vartheta\|_{L^2(\Omega)}+\Vert \psi \Vert_{L^2(\Gamma)}).
\end{align*}
Moreover, if $\psi_n $ converges weakly to $\psi$ in $L^2(\Gamma)$ then $\xi_{n}$ converges strongly to $\xi$ in $H^1(\Omega)\cap C(\bar \Omega)$ as $n \to \infty$ .
\end{lemma}

In what follows, we will write $\phi_n \rightharpoonup \phi$ and $\phi_n \to \phi$ when $\phi_n$ converges weakly to $\phi$  and $\phi_n$ converges  strongly to $\phi$, respectively.

By Lemma \ref{lm1}, the equation \eqref{P2} has a unique solution $y \in H^1(\Omega)\cap C(\bar \Omega)$ for any $u,\lambda^{(1)} \in L^2(\Gamma)$. Additionally, there exists a constant $C_0>0$ which does not depend on $f, u, \lambda^{(1)}$ so that the following relation is true:
\begin{align*}
\Vert y\Vert_{H^1(\Omega)}+\Vert y \Vert_{C(\bar \Omega)} \leq C_0  \Vert u + \lambda^{(1)} \Vert_{L^2(\Gamma)}.
\end{align*}

The following properties of $K(\lambda)$ can be showed similarly to those found in \cite{Son1} with some minor modifications. For the convenience of the reader, however, we are going to provide some brief proofs here. 
\begin{lemma}\label{lm2}
Under the assumptions {\rm (A4)-(A5)}, $K(\lambda)$ is a nonempty and closed set for any $\lambda \in \Lambda$.
\end{lemma}
\begin{proof}
Taking any $\lambda=(\lambda^{(1)},\lambda^{(2)})\in \Lambda$, the following equation:
\begin{align*}
\begin{cases}
Ay  + f(x,y)= 0  \quad &{\rm in}\ \Omega,\\
\partial_{n_A} y +g(x,y)= \lambda^{(1)}-\lambda^{(2)} \quad  &{\rm on}\ \Gamma
\end{cases}
\end{align*}
 has a unique solution $y_0 \in Y$. Putting $u_0=-g(\cdot,y_0)-\lambda^{(2)}$, one obtains
$(y_0, u_0)\in K(\lambda)$. Hence $K(\lambda) \neq \emptyset$. 

Let $\{(y_n, u_n)\}$ be a sequence in $K(\lambda)$ such that $ (y_n, u_n) \to (y,u)$ in $Y \times U$ as $n \to \infty$. By the continuous property of $f$ and $g$ with respect to $y$ and the Lebesgue dominated convergence theorem, we have $z=(y,u) \in K(\lambda)$. Hence, the proof of Lemma \ref{lm2} is complete.
\end{proof}


\begin{lemma} \label{lm3}
Suppose that the assumptions {\rm (A4)-(A5)} are satisfied and $\{\lambda_n \}$ is a sequence such that $\lambda_n \to \hat \lambda$ in $\Lambda$. Then for any $(\hat y, \hat u) \in K(\hat \lambda)$, there exists a sequence $\{(y_n, u_n)\}$ with $(y_n, u_n)\in K(\lambda_n)$ fulfilling $(y_n, u_n) \to (\hat y, \hat u)$ in $Y \times U$.
\end{lemma}

\begin{proof} Let $\lambda_n=(\lambda^{(1)}_n, \lambda^{(2)}_n)\to \hat \lambda=(\hat \lambda^{(1)},\hat \lambda^{(2)})$ and fix $(\hat y, \hat u) \in K(\hat \lambda)$. 
	
	From the assumption $(A5)$ with $M=\|\hat y \|_{C(\bar \Omega)}$ and a Taylor expansion it follows that
$g(\cdot,\hat y)\in L^\infty(\Gamma)$ and $\hat u+ \hat \lambda^{(2)}+g(x,\hat y)+ \lambda^{(1)}_n-\lambda^{(2)}_n \in L^2(\Gamma)$. For this reason, the following equation possesses a unique solution:
\begin{align}\label{lm3e2}
\begin{cases}
Ay  + f(x,y)= 0  \quad &{\rm in}\ \Omega,\\
\partial_{n_A} y+ g(x,y) = \hat u+ \hat \lambda^{(2)}+g(x,\hat y)+ \lambda^{(1)}_n-\lambda^{(2)}_n \quad  &{\rm on}\ \Gamma.
\end{cases}
\end{align}
Let $y_n$ be a unique solution of the equation \eqref{lm3e2}. Setting $u_n=\hat u+g(x,\hat y)+\hat \lambda^{(2)}-g(x,y_n)-\lambda^{(2)}_n$ we can conclude $(y_n,u_n)\in K(\lambda_n)$. 

Denoting $\omega_n:=y_n-\hat y$ one derives 
\begin{align*}
\begin{cases}
A\omega_n  + f(x,y_n)-f(x,\hat y)= 0  \quad &{\rm in}\ \Omega,\\
\partial_{n_A} \omega_n+g(x,y_n)-g(x,\hat y) =  \lambda^{(1)}_n-\hat \lambda^{(1)}-\lambda^{(2)}_n+\hat \lambda^{(2)} \quad  &{\rm on}\ \Gamma.
\end{cases}
\end{align*}
Hence, there exist $ y_{n1} =\hat y+\theta_{1n} (y_n-\hat y)$ and
$\hat y_{n2}= \hat y+\theta_{2n} (y_n-\hat y)$ with $\theta_{1n},
\theta_{2n} \in (0,1)$ so that the above equation can be expressed again by
\begin{align*}
\begin{cases}
A\omega_n  + f_y(x, y_{n1})\omega_n= 0  \quad &{\rm in}\ \Omega,\\
\partial_{n_A} \omega_n +g_y(x,y_{n2})\omega_n=  \lambda^{(1)}_n-\hat \lambda^{(1)}-\lambda^{(2)}_n+\hat \lambda^{(2)} \quad  &{\rm on}\ \Gamma.
\end{cases}
\end{align*}
 Employing Lemma \ref{lm1} we may state that there is a constant $C_0>0$ independent of $f_y(x, y_{n1})$, $g_y(x, y_{n2})$, $\lambda_n$ and
$\hat \lambda$ such that the following estimate holds:
\begin{align*}
\Vert \omega_n \Vert_Y \leq C_0 \Vert \lambda^{(1)}_n-\hat \lambda^{(1)}-\lambda^{(2)}_n+\hat \lambda^{(2)} \Vert_{L^2(\Gamma)} 
\leq C_0 \Vert \lambda_n-\hat \lambda \Vert_\Lambda.
\end{align*} This means that $\Vert y_n-\hat y \Vert_Y \leq C_0 \Vert \lambda_n-\hat \lambda \Vert_\Lambda. $
Since $\lambda_n \to \hat \lambda$ in $\Lambda$, $y_n \to \hat y$ in
$Y$. Thanks to the assumption $(A5)$ and the Lebesgue dominated convergence theorem, we
obtain $g(x, y_n)\to g(x, \hat y)$ in $L^2(\Gamma)$ and so $u_n \to \hat u$ in $L^2(\Gamma)$. Conclusion, the proof of Lemma \ref{lm3} is complete.
\end{proof}


\section{Proof of the main result} \label{Sec.4}
Let us devote to the proof of Theorem \ref{Theo1} in this section. To get started, for each $(\mu, \lambda)\in \Pi \times \Lambda$ with $\|\mu -\bar \mu \|_{\Pi} +\|\lambda -\bar \lambda \|_\Lambda\leq \epsilon_0$ we denote by $V=V (\mu, \lambda)$ the optimal value function of problem $P(\mu,\lambda)$, that is,
$$V (\mu, \lambda)= \inf \limits_{(y,u)\in K(\lambda)} F(y,u,\mu). $$
At the begining, we are going to prove the first statement of Theorem \ref{Theo1}.\\

\noindent \textbf{(i) The non-emptiness of $S(\mu, \lambda)$}

Firstly, take any $(\mu, \lambda)\in \Pi \times \Lambda$ with $\|\mu -\bar \mu \|_{\Pi} +\|\lambda -\bar \lambda \|_\Lambda\leq \epsilon_0$. Applying Lemma \ref{lm2}, one deduces $K(\lambda)\neq \emptyset$. Then, we take $(y,u) \in K(\lambda)$ and put $M:=\| y \|_{C(\bar \Omega)}$. Because of the assumption (A1), there exist functions $r_{1M}\in L^1(\Omega), r_{2M} \in L^1(\Gamma)$ such that
\begin{align*}
|J(y,u,\mu)|\leq \| r_{1M} \|_{L^1(\Omega)}+\|r_{2M} \|_{L^1(\Gamma)}+k_{\rm max}\| u\|^2_{L^2(\Gamma)}<+\infty.
\end{align*}
where $k_{\rm max}$ is defined by
$$k_{\rm max}:=\max_{[\|\bar \mu\|_\Pi-\epsilon_0,\|\bar \mu\|_\Pi+\epsilon_0]}\varphi(t). $$
By the definition of $V$, there exists a sequence $\{ (y_n, u_n) \} \subset K(\lambda)$ so that it holds
\begin{align*}
V(\mu,\lambda)= \lim_{n \to \infty}J(y_n, u_n,\mu).
\end{align*}
Consequently, one has
$$
1+ V(\mu, \lambda) > J(y_n,u_n,\mu)
$$
for any $n\geq n_0$, where $n_0>0$ is a sufficiently large number. Combining this with the assumption (A2), we obtain
 $$
 1+ V(\mu, \lambda) > \int_\Omega r_{1}(x)dx+ \int_\Gamma \big(r_{2}(x)+\gamma u_n^2(x)\big)d\sigma, \quad \forall n\geq n_0,
 $$
which entails $\|u_n\|_{L^2(\Gamma)}\leq M_0$ for all $n\ge n_0$ with a suitable constant $M_0>0$. Due to the fact that $L^2(\Gamma)$ is reflexive, there is a subsequence $\{u_{n_k}\}$ satisfying $u_{n_k}\rightharpoonup \hat u$ in $L^2(\Gamma)$. Let us now denote by $\hat y$ the solution to the following equation:
\begin{align}\label{lm3e41}
	\begin{cases}
		Ay + f(x,y)= 0  \quad &{\rm in}\ \Omega,\\
		\partial_{n_A} y = \hat  u + \lambda^{(1)} \quad  &{\rm on}\ \Gamma.
	\end{cases}
\end{align}
Since $(y_{n_k}, u_{n_k}) \in K(\lambda)$, one has 
\begin{align*}
	\begin{cases}
		Ay_{n_k} + f(x,y_{n_k})= 0  \quad &{\rm in}\ \Omega,\\
		\partial_{n_A} y_{n_k} =   u_{n_k} + \lambda^{(1)} \quad  &{\rm on}\ \Gamma.
	\end{cases}
\end{align*}
On the one hand, noticing $u_{n_k}+\lambda^{(1)} \rightharpoonup \hat u +\lambda^{(1)}$ in $U$ it follows from Lemma \ref{lm1} that $y_{n_k}\to \hat y$ in $Y$. By the assumption (A5), we derive $g(\cdot,y_{n_k})\to g(\cdot,\hat y(x))$ in $L^2(\Gamma)$. Therefore $g(\cdot,y_{n_k})+u_{n_k}+\lambda^{(2)}\rightharpoonup g(\cdot,\hat y)+\hat u+\lambda^{(2)}$ in $L^2(\Gamma)$. On the other hand, the set $\{w \in L^2(\Gamma):\ w(x)\leq 0 \quad \text{a.a. } x\in \Gamma \}$ is weakly closed in $L^2(\Gamma)$ because it is convex and closed. As a result, one realizes
\begin{equation*}
g(x,\hat y)+\hat u+\lambda^{(2)}\leq 0 \text{ a.a. } x\in \Gamma
\end{equation*}
thanks to $g(x,y_{n_k})+u_{n_k}+\lambda^{(2)}\leq 0 \text{ a.a. } x\in \Gamma$. For this reason, we link the equation \eqref{lm3e41} to conclude $(\hat y, \hat u) \in K(\lambda)$. 

Besides, from the assumption (A3) and \cite[Theorem 3.23]{Dacorogna} one can claim that $J(\cdot,\cdot,\mu)$ is weakly lower semicontinuous for each $\mu \in \Pi$. Therefore, we arrive at
$$J(\hat y, \hat u, \mu) \leq \liminf \limits_{k \to +\infty} J(y_{n_k}, u_{n_k},\mu)=\lim \limits_{k \to +\infty} J(y_{n_k}, u_{n_k},\mu)=V(\mu,\lambda),$$
which implies $J(\hat y, \hat u, \mu)=V(\mu, \lambda)$. This means $(\hat y, \hat u)\in S(\mu, \lambda)$ to finish our proof.

Next, let us come back to prove the second statement of Theorem \ref{Theo1}.\\

\noindent \textbf{(ii) Upper semicontinuity of $S(\cdot, \cdot)$}

Suppose that the assertion (ii) of Theorem \ref{Theo1} is false. Then there exist open sets $V_1$ in $Y$, $V_2$ in $U$ and sequences $\{(\mu_n, \lambda_n)\} \subset \Pi \times \Lambda$, $\{(\bar y_n,\bar u_n)\} \subset Y\times U$ such that
\begin{align} \label{assum}
\begin{cases}
(\mu_n, \lambda_n) \to (\bar \mu,\bar \lambda), \quad \|\mu_n - \mu \|_\Pi +\|\lambda_n -\bar \lambda\|_\Lambda \leq \epsilon_0, \\
S(\bar \mu, \bar \lambda) \subset V_1\times V_2,  \\
(\bar y_n,\bar u_n) \in S(\mu_n, \lambda_n) \setminus V_1\times V_2, \ \forall n \geq 1.
\end{cases}
\end{align}
\textit{Our next strategy is as follows:} If we can show after choosing a subsequence that  $(\bar y_{n},\bar u_{n})\to (\bar y, \bar u)\in S(\bar \mu, \bar \lambda)$ in $Y\times U$ as $n \to +\infty$, then $(\bar y_{n}, \bar u_{n})\in W_1 \times W_2$ for $n$ large enough. This contradicts to \eqref{assum}. Hence, the proof can be done.

In order to demonstrate this, the following lemmas play a vital role in our verification.

\begin{lemma}\label{converge} There is a subsequence of $\{(\bar y_n,\bar u_n)\}$, also denoted by $\{(\bar y_n,\bar u_n)\}$, such that 
$$ \bar y_{n} \to \bar y \ \text{in}\ Y \quad \text{and}\quad \bar u_{n} \rightharpoonup \bar u\ \text{in} \ U$$
for some $(\bar y, \bar u)\in S(\bar \mu, \bar \lambda).$
\end{lemma}
\begin{proof} In the first step, assume that $y_n$ be the solution of the following equation: 
\begin{align*} \begin{cases}
	Ay  + f(x,y)= 0  \quad &{\rm in}\ \Omega,\\
	\partial_{n_A} y+g(x,y) = \lambda_n^{(1)}-\lambda_n^{(2)} \quad  &{\rm on}\ \Gamma,
\end{cases}	
\end{align*}
and $u_n:=-g(x,y_n)-\lambda_n^{(2)}$. Using Lemma \ref{lm1}, we have 
\begin{align*}
\|y_n\|_Y &\leq C_0\|\lambda_n^{(1)}-\lambda_n^{(2)} \|_{L^2(\Gamma)}\\
&\leq C_0(\|\lambda_n^{(1)}-\bar \lambda_n^{(1)} \|_{L^2(\Gamma)}+\|\bar \lambda_n^{(1)}-\bar \lambda_n^{(2)} \|_{L^2(\Gamma)}+\|\bar \lambda_n^{(2)}-\lambda_n^{(2)} \|_{L^2(\Gamma)})\\
&\leq C_0(\|\lambda_n-\bar \lambda\|_{\Lambda}+\|\bar \lambda_n^{(1)}-\bar \lambda_n^{(2)} \|_{L^2(\Gamma)})\\
&\leq C_0(\epsilon_0+2\|\bar \lambda \|_\Lambda),
\end{align*}
which is equivalent to $\|y_n\|_Y\leq M_1$ for all $n$ with $M_1:=C_0(\epsilon_0+2\|\bar \lambda \|_\Lambda)$. Consequently, it follows from the assumption (A5) that 
\begin{align*}\|u_n\|_U = \|g(\cdot,y_n)+\lambda_n^{(2)}\|_U&\leq \|g(\cdot,y_n)-g(\cdot,0)\|_U+\|\lambda_n^{(2)}\|_U\\
&\leq C_{gM_1}\|y_n\|_{L^2(\Gamma)}+\|\lambda_n\|_\Lambda\\
&\leq C_{gM_1}\|y_n\|_Y|\Gamma|+\|\bar \lambda\|_\Lambda+\epsilon_0\\
&\leq C_{gM_1}M_1|\Gamma|+\|\bar \lambda\|_\Lambda+\epsilon_0,
\end{align*}
where $|\Gamma|$ is the volume of $\Gamma$. Hence, one achieves $\|u_n\|\leq M_2$ for all $n$ with $M_2:=C_{gM_1}M_1|\Gamma|+\|\bar \lambda\|_\Lambda+\epsilon_0$. Because of the assumption (A1), there exist $r_{1M_1}\in L^1(\Omega)$ and $r_{2M_1}\in L^1(\Gamma)$ such that 
\begin{align*}
	J(y_n, u_n,\mu_n) &\leq \|r_{1M_1}\|_{L^1(\Omega)}+\|r_{2M_1}\|_{L^1(\Gamma)}+k_{\rm max} \|u_n\|_U \\
	&\leq \|r_{1M_1}\|_{L^1(\Omega)}+\|r_{2M_1}\|_{L^1(\Gamma)}+k_{\rm max} M_2=:M_3, \quad \forall n.
	\end{align*}
Moreover, from the definitions of $y_n$ and $u_n$ we may conclude $(y_n,u_n) \in Y\times U$ and 
\begin{align*} 
	&\begin{cases}
		Ay_n  + f(x,y_n)= 0  \quad &{\rm in}\ \Omega,\\
		\partial_{n_A} y_n = u_n+\lambda_n^{(1)} \quad  &{\rm on}\ \Gamma,
	\end{cases}	\\
	&u_n+g(x,y_n)+\lambda_n^{(2)}=0. 
\end{align*}
Hence, we have $(y_n,u_n)\in K(\lambda_n)$ for all $n$. Combining this with the fact $(\bar y_n,\bar u_n)\in S(\mu_n,\lambda_n)$ yields  
$J(\bar y_n, \bar u_n,\mu_n) \leq J(y_n, u_n, \mu_n)$. For this reason, one gets
$$J(\bar y_n, \bar u_n,\mu_n) \leq M_3, \quad \forall n.$$
In other words, the sequence $\{V(\mu_n, \lambda_n)\}$ is bounded. Besides, the assumptions (A2) and (A3) imply
$$J(\bar y_n, \bar u_n,\mu_n)\geq \|r_1\|_{L^1(\Omega)}	+\|r_2\|_{L^1(\Gamma)}+\gamma \|\bar u_n\|^2_U, \quad \forall n.$$ 
Therefore, there exists a number $M_4>0$ such that $\|\bar u_n\|_U\leq M_4$ for all $n$. Since $U$ is reflexive, by passing a subsequence, we can assume $\bar u_{n}\rightharpoonup \bar u$ in $U$.

In the second step, let $\bar y\in Y$ be a unique solution to the following equation: 
\begin{align*}
	\begin{cases}
		Ay + f(x,y)= 0  \quad &{\rm in}\ \Omega,\\
		\partial_{n} y = \bar  u + \bar \lambda^{(1)} \quad  &{\rm on}\ \Gamma.
	\end{cases}
\end{align*}
Since $\lambda_n^{(1)}\to \bar \lambda^{(1)} $ in $U$, it is clear that $\bar  u_n + \lambda_n^{(1)}\rightharpoonup \bar  u + \bar \lambda^{(1)} $ in $U$. The employment of Lemma \ref{lm1} gives $\bar y_n \to \bar y$ in $Y$. 
It follows that $g(\cdot, \bar y_n)\to g(\cdot, \bar y)$ in $U$ and so 
$g(\cdot, \bar y_n)+\bar u_n+\lambda_n^{(2)}\rightharpoonup  g(\cdot, \bar y)+\bar u+\bar \lambda^{(2)}$ in $U$. Due to the facts $g(x, \bar y_n)+\bar u_n+\lambda_n^{(2)}\leq 0$ a.a. $x\in \Gamma$ and the set $\{w \in L^2(\Gamma):\ w(x)\leq 0 \quad \text{a.a. } x\in \Gamma \}$ is weakly closed in $U$, one finds
$$g(x, \bar y)+\bar u+\bar \lambda^{(2)}\leq 0 \text{ a.a. } x\in \Gamma.$$
Therefore, it holds $(\bar y, \bar u)\in K(\bar \lambda)$. 
    
Next, it remains to prove that $(\bar y, \bar u) \in S(\bar \mu, \bar \lambda)$. Indeed, taking any $(z, v) \in K(\bar \lambda)$, we only need to show that 
\begin{align}\label{sol_result}
J(\bar y, \bar u, \bar \mu) \leq J(z,v, \bar \mu).
\end{align}
On the one hand, by Lemma \ref{lm3} there is a sequence $\{(z_n, v_n)\} \subset K(\lambda_n)$ satisfying $z_n \to z$ and $v_n\to v$. Furthermore, from $(\bar y_n, \bar u_n) \in S(\mu_n,\lambda_n)$ we have
\begin{align}\label{sol1}
J(\bar y_n,\bar u_n, \mu_n)\leq J(z_n, v_n, \mu_n).
\end{align}
Theorem 3.23 in \cite{Dacorogna} shows that $J(\cdot,\cdot,\cdot)$ is lower semicontinuous at $(\bar y, \bar u, \bar \mu)$ and this implies that 
\begin{align*}
J(\bar y, \bar u, \bar \mu) \leq \liminf \limits_{n\to +\infty} J(y_n, u_n,\mu_n).
\end{align*}
From \eqref{sol1} it entails
\begin{align} \label{sol2}
J(\bar y, \bar u, \bar \mu) \leq \liminf \limits_{n\to +\infty} J(z_n, v_n,\mu_n).
\end{align}
On the other hand, one gets
$$J(z_n, v_n,\mu_n) = \int_{\Omega}L(x,z_n(x),\mu^{(1)}_n(x))dx+\int_{\Gamma}\big(\ell(x,z_n(x),\mu^{(2)}_n(x))+ \varphi (\mu^{(2)}_n(x))v^2_n(x)\big) d\sigma.$$
Since $(z_n, v_n)\to (z,v)$ in $C(\bar \Omega)\times L^2(\Gamma)$, there exist a constant $M_5>0$, a function $v_0 \in L^2(\Gamma)$ such that after choosing a subsequence, for every $n$ we derive
$$\|z_n\|_{C(\bar \Omega)}\leq M_5,\ |v_n(x)|\leq v_0(x) \quad \text{a.a.}\ x\in \Gamma.$$
By the continuity of $L(x,\cdot, \cdot)$ and $\ell(x,\cdot,\cdot,\cdot)$, we have
\begin{align*}
&L(x,z_n(x), \mu^{(1)}_n(x))\to L(x,z(x),\bar \mu^{(1)}(x)) \quad \text{a.a.} \ x\in \Omega,\\
&\ell(x', z_n(x'),\mu^{(2)}_n(x')) +\varphi(\mu^{(2)})v^2_n\to \ell(x', z(x'), \bar \mu^{(2)}(x')) +\varphi(\mu^{(2)})v(x') \quad \text{a.a.} \ x' \in \Gamma.
\end{align*}
Moreover, by the assumptions (A1) and (A3), there are functions $r_{1M}\in L^1(\Omega)$, $r_{2M}\in L^1(\Gamma)$ so that one can arrive at
\begin{align*}
&|L(x,z_n(x),\mu^{(1)}_n(x))| \leq r_{1M_5}(x) \quad \text{a.a.} \ x\in \Omega,\\
&|\ell(x',z_n(x'),\mu^{(2)}_n(x'))| \leq
r_{2M_5}(x'), \quad |\varphi(\mu^{(2)})v^2_n(x')|\leq k_{\rm max} v^2_0(x') \quad \text{a.a.} \ x' \in \Gamma,
\end{align*}
for all $n\geq 1$. Hence, applying the Lesbegue dominated convergence theorem leads to
\begin{align*}
\lim \limits_{n \to +\infty} J(z_n, v_n, \mu_n)&= \int_{\Omega}L(x,z(x),\bar \mu^{(1)}(x))dx+\int_{\Gamma}\big(\ell(x,z(x),\bar \mu^{(2)}(x))+\varphi(\bar \mu^{(2)})v^2(x) \big )d\sigma \\&=J(z,v,\bar \mu).
\end{align*}
 From this and \eqref{sol2}, we have shown \eqref{sol_result}. All in all, the proof of Lemma \ref{converge} is complete.
 \end{proof}

\begin{lemma} \label{convergence}
$\bar u_n \to \bar u$  in $U$.
\end{lemma}
\begin{proof}
According to Lemma 2.4 in \cite{Son1}, since $(\bar y_n,\bar  u_n)\in S(\mu_n, \lambda_n)$ and $(\bar y, \bar u)\in S(\bar \mu,\bar \lambda)$, there exist function $\phi_n, \bar \phi \in H^1(\Omega)\cap C(\bar \Omega)$, and $e_n, e \in L^2(\Gamma)$ satisfying the following conditions:
\begin{align} \label{eq1}
\begin{cases}
A^*\bar \phi  + f_y(\cdot,\bar y)\bar \phi= L_y(\cdot, \bar y,\bar \mu^{(1)})  \quad &{\rm in}\ \Omega,\\
\partial_{n_{A^*}} \bar \phi+g_y(\cdot, \bar y)\bar \phi = \ell_y(\cdot,\bar y, \bar \mu^{(2)}) - 2\varphi(\bar \mu^{(2)})\bar u\, g_y(\cdot, \bar y)\quad  &{\rm on}\ \Gamma,
\end{cases}
\end{align}
\begin{align} \label{eq2}
\begin{cases}
A^*\phi_n  + f_y(\cdot,\bar y_n)\phi_n = L_y(\cdot,\bar y_n, \mu^{(1)}_n)  \quad &{\rm in}\ \Omega,\\
\partial_{n_{A^*}} \phi_n+g_y(\cdot,\bar y_n)\phi_n = \ell_y(\cdot,\bar y_n,\mu^{(2)}_n) - 2\varphi(\mu^{(2)}_n)\bar u_n\, g_y(\cdot,\bar  y_n)\quad  &{\rm on}\ \Gamma,
\end{cases}
\end{align}
and
\begin{align}
	&\langle \bar \phi+ 2\varphi(\bar \mu^{(2)})\bar u , G(\bar y,\bar u,\bar \lambda)-G(\bar y_n,\bar u_n,\lambda_n)\rangle  \leq 0,  \label{eq5} \\
	&\langle \phi_n+ 2\varphi(\mu^{(2)}_n)\bar u_n , G(\bar y_n,\bar u_n,\lambda_n)-G(\bar y,\bar u,\bar \lambda)\rangle  \leq 0,  \label{eq6}
\end{align}
where  $G(y,u,\lambda^{(2)}):=g(\cdot, y)+u+ \lambda^{(2) }$ and 
$$ \langle w_1, w_2\rangle:= \int_{\Gamma} w_1(x)w_2(x)d\sigma. $$
Our proof is separated into two steps as follows:

{\it Claim 1.}\quad There is a subsequence of $\{\phi_n \} $ converges strongly to $\phi_0$ in $L^2(\Omega)$ for some $\phi_0 \in H^1(\Omega)$. 

Indeed, since $\bar y_n \to \bar y$ in $Y$ and $\bar u_n \rightharpoonup \bar u$ in $U$, there is $M>0$ such that
\begin{align*}
\|\bar y\|_Y,\ \|\bar y_n\|_Y, \ \|\bar u\|_U,\  \|\bar u_n\|_U \leq M_6
\end{align*}
for all $n\geq 1$. By the assumptions (A1) and (A5), there exist functions $r_{3M_6} \in L^\infty(\Omega), \ r_{4M_6}\in L^\infty(\Gamma)$ and positive constants $C_{\ell M_6}, C_{gM_6}$ so that one gains the following estimates:
\begin{align}
&|L_y(x,\bar y_n, \mu^{(1)}_n)|\leq r_{3M_6}(x), \notag \\
&|\ell_y(x',\bar y_n,\mu^{(2)}_n)|\leq C_{\ell M_6}|\bar y_n|+r_{4M}(x'), \label{M1} \\
&|g_y(x', \bar y_n)| \leq C_{gM_6}, \notag
\end{align}
for all $n\geq 1$, $x\in \Omega$ and $x'\in \Gamma$. From \eqref{eq2} and \eqref{M1}, we apply Lemma \ref{lm1} to get
\begin{align*}
\|\phi_n\|_Y &\leq \|L_y(\cdot,\bar y_n, \mu^{(1)}_n)\|_{L^2(\Omega)}+\|\ell_y(\cdot,\bar y_n,\mu^{(2)}_n)-2\varphi(\mu^{(2)}_n)\bar u_n\,g_y(\cdot,\bar y_n)\|_{L^2(\Gamma)}\\
& \leq \|r_{3M_6}\|_{L^\infty(\Omega)}|\Omega|+(1+C_{gM_6})\Big((C_{\ell M_6}M_6+\|r_{4M_6}\|_{L^\infty(\Gamma)})|\Gamma|+2k_{\rm  max} \|\bar u_n\|_{L^2(\Gamma)}\Big)\\
& \leq \|r_{3M_6}\|_{L^\infty(\Omega)}|\Omega|+(1+C_{gM_6})\Big((C_{\ell M_6}M_6+\|r_{4M_6}\|_{L^\infty(\Gamma)})|\Gamma|+2k_{\rm max} M_6\Big) ,
\end{align*}
where $|\Omega|$ and $|\Gamma|$ are the volumes of $\Omega$ and
$\Gamma$, respectively. Thus, it follows that $\{\phi_n\}$ is bounded in
$H^1(\Omega)$. By passing a subsequence, we can assume that $\phi_n
\rightharpoonup \phi_0$ in $H^1(\Omega)$. Since the embedding
$H^1(\Omega) \hookrightarrow L^2(\Omega)$ is compact, it implies $\phi_n
\to \phi_0$ in $L^2(\Omega)$. Therefore, Claim 1 is proved.

{\it Claim 2.}\quad $\bar u_n \to \bar u$ in $U$.

Summing up \eqref{eq5} and \eqref{eq6} gives
\begin{align}
&\langle \phi_n-\bar \phi, G(\bar y_n,\bar u_n,\lambda_n)-G(\bar y,\bar u,\bar \lambda)\rangle \notag \\
&\quad + 2\langle \varphi(\mu^{(2)}_n)\bar u_n  - \varphi(\bar \mu^{(2)})\bar u, g(\cdot,\bar y_n)+\lambda^{(2)}_n)-g(\cdot, \bar y) -\bar \lambda^{(2)}  \rangle \notag \\
&\quad + 2\langle (\varphi(\mu^{(2)}_n)  - \varphi(\bar \mu^{(2)}))\bar u_n, \bar u_n-\bar u \rangle + 2\langle \varphi(\bar \mu^{(2)})(\bar u_n  - \bar u),\bar u_n-\bar u \rangle \leq 0. \label{eq7}
\end{align}
Since $\bar y_n \to \bar y$ in $Y$, $\lambda_n \to \bar  \lambda$ and $\bar u_n \rightharpoonup \bar u$ in $L^2(\Gamma)$,
we get $g(\cdot,\bar y_n)+\lambda^{(2)}_n \to g(\cdot, \bar y)+\bar \lambda^{(2)}$ and $G(\bar y_n,\bar u_n, \lambda_n)\rightharpoonup G(\bar y,\bar u,\bar \lambda)$ in $L^2(\Gamma)$. Moreover, one also has
\begin{align*}
	\| \varphi(\mu^{(2)}_n)\bar u_n  - \varphi(\bar \mu^{(2)})\bar u\|_U \leq \| \varphi(\mu^{(2)}_n)\bar u_n\|_U + \| \varphi(\bar \mu^{(2)})\bar u\|_U \leq 2k_{\rm max}M, 
\end{align*} and $ \phi_n - \bar \phi \to \phi_0-\bar \phi$ in $L^2(\Gamma)$. As a consequence, they hold
\begin{align}
&\lim_{n\to +\infty} \langle \phi_n-\bar \phi, G(\bar y_n,\bar u_n,\lambda_n)-G(\bar y,\bar u,\bar \lambda)\rangle =0 \label{term1},\\
&\lim_{n\to +\infty} \langle \varphi(\mu^{(2)}_n)\bar u_n  - \varphi(\bar \mu^{(2)})\bar u, g(\cdot, \bar y_n)+\lambda^{(2)}_n)-g(\cdot, \bar y) -\bar \lambda^{(2)}  \rangle  =0 \label{term2}. 
\end{align}
Moreover, it follows from the assumption (A3) that
\begin{align*}
|\langle (\varphi(\mu^{(2)}_n)  - \varphi(\bar \mu^{(2)}))\bar u_n, \bar u_n-\bar u \rangle|& \leq  \langle |\varphi(\mu^{(2)}_n)  - \varphi(\bar \mu^{(2)}|.|\bar u_n|, |\bar u_n-\bar u| \rangle \\
& \leq  \langle k_\varphi |\mu^{(2)}_n)  - \bar \mu^{(2)}|^\theta.|\bar u_n|, |\bar u_n-\bar u| \rangle \\
&\leq k_\varphi \|\mu^{(2)}_n)  - \bar \mu^{(2)}\|^\theta_{L^{\infty}(\Gamma)}. \|\bar u_n\|_U. \|\bar u_n-\bar u \|_U\\
&\leq 2M^2.k_\varphi \|\mu^{(2)}_n)  - \bar \mu^{(2)}\|^\theta_{L^{\infty}(\Gamma)}.  
\end{align*}
This implies that 
\begin{align}\label{term3}
	\lim_{n\to +\infty} \langle (\varphi(\mu^{(2)}_n)  - \varphi(\bar \mu^{(2)}))\bar u_n, \bar u_n-\bar u \rangle =0.
\end{align}
On the other hand, by the assumption (A3) we have
\begin{align} \label{term4}
	\langle \varphi(\bar \mu^{(2)})(\bar u_n  - \bar u), \bar u_n-\bar u \rangle = \int_\Gamma \varphi(\bar \mu^{(2)}(x))(\bar u_n(x)-\bar u(x))^2 d\sigma \geq \gamma \|\bar u_n-\bar u\|^2_U.	 
\end{align}
Collecting from \eqref{eq7} to \eqref{term4}, we may assert that 
\begin{align*} \label{eq8}
\limsup \limits_{n\to +\infty} \gamma \|\bar u_n-\bar u\|^2_U \leq \limsup \limits_{n\to +\infty} \langle \varphi(\bar \mu^{(2)})(\bar u_n  - \bar u), \bar u_n-\bar u \rangle \leq 0,
\end{align*}
which follows that $\bar u_n \to \bar u$ in $L^2(\Gamma)$. For this reason, Claim 2 is indicated. This completes the proof of Lemma \ref{convergence}.
\end{proof}

From Lemmas \ref{converge} and \ref{convergence}, we obtain the assertion (ii) of Theorem \ref{Theo1}. Finally, following a similar argument to the proof of \cite[Theorem 1.1]{Son1}, we obtain the continuity of $S(\cdot, \cdot)$ at $(\bar \mu,\bar \lambda)$ when $S(\bar \mu, \bar \lambda)$ is singleton. Summarizing, the proof of Theorem \ref{Theo1} is complete.

\section*{Acknowledgments}
 The research of the first author is partially funded by  Vietnam Institute for Advanced Study in Mathematics (VIASM). This research of the second author (Tuan Anh Dao) is funded (or partially funded) by the Simons Foundation Grant Targeted for Institute of Mathematics, Vietnam Academy of Science and Technology.

\end{document}